\documentclass[12 pt]{amsart}
\usepackage{amsmath,amsthm,amscd,enumerate,amssymb}
\newtheorem{theorem}{Theorem}[section]

\theoremstyle{definition}

\numberwithin{definition}{section}

\theoremstyle{remark}

\begin{document}
\title[Sum of $r$'th roots]{The sum of the $r$'th roots of \\first $n$ natural numbers and new formula for factorial}
\date{}
\author{Snehal Shekatkar}
\address{Dept.of Physics\\Indian institute of science education and research\\ Pune, India\\Pin:411021}
\email{snehalshekatkar@gmail.com}

\begin{abstract} 
\ Using the simple properties of Riemman integrable functions, Ramanujan's formula for sum of the square roots of first $n$ natural numbers has been generalized to include $r'th$ roots where $r$ is any real number greater than 1. As an application we derive formula that gives factorial of positive integer $n$ similar to Stirling's formula.
\end{abstract}

\maketitle
         The formula for the sum of the square roots of first n natural numbers has been given by Srinivas Ramanujan (\cite{Ra}). Here we extend his result to the case of $r'$th roots, where $r$ is a real number greater than 1. \\
\textbf{Statement of result:}
\begin{theorem}
\ Let $r$ be a real number with $r\geq 1$ and $n$ be a positive integer. Then
\begin{equation}
\sum_{x=1}^{x=n}x^{\frac{1}{r}}\quad = \quad\frac{r}{r+1}(n+1)^{\frac{1+r}{r}}-\frac{1}{2}(n+1)^{\frac{1}{r}}-\phi_{n}(r)
\end{equation}
 where $\phi_{n}$ is a function of $r$ with $n$ as a parameter. This function is bounded between $0$ and $\frac{1}{2}$.
\end{theorem}
\begin{proof}
 For a closed interval [a,b] we define partition of this interval as a set of points {$x_{0}=a,x_{1},...,x_{n-1}=b$} where $x_{i} < x_{j}$ whenever $i < j$ . Now consider the closed interval [0,n] and consider a partition $P$ of this interval,where $P$ is a set
$\{0,1,2,...,n\}$.\\

 Consider a function defined as $f(x)$=$x^{\frac{1}{r}}$.\\
We have,\\
$$I=\int_{0}^{n}f(x)dx=\lim_{\Delta x_{i}\rightarrow 0} \sum_{i=0}^{n-1} f(x_{i})\Delta x_{i}$$
where $$\Delta x_{i}=x_{i+1}-x_{i}$$
\ We define lower sum for partition $P$ as:
$$L=\sum_{i=0}^{n-1}f(i)\Delta x_{i}=\sum_{i=0}^{n-1}i^{\frac{1}{r}} $$
Similarly,upper sum for $P$ is
$$U=\sum_{i=1}^{n}f(i)\Delta x_{i}=\sum_{i=0}^{n-1}(i+1)^{\frac{1}{r}}$$
\ We write value of integral $I$ as average of $L$ and $U$ with some correction term.
$$2I= L+U+\phi$$
$$\therefore 2\int_{0}^{n}x^{\frac{1}{r}}dx=\sum_{i=0}^{n-1}\bigl(i^{\frac{1}{r}}+(i+1)^{\frac{1}{r}}\bigr)+\phi$$ 
$$\therefore \frac{2r}{r+1}\bigl(x^{\frac{1+r}{r}}\bigr)_{0}^{n}=0^{\frac{1}{r}}+2\sum_{i=1}^{n-1}i^{\frac{1}{r}}+n^{\frac{1}{r}}+\phi$$
$$\therefore \sum_{i=1}^{n-1}i^{\frac{1}{r}}=\frac{r}{r+1}n^{\frac{1+r}{r}}-\frac{n^{\frac{1}{r}}}{2}-\phi$$
\ where the term of $\frac{1}{2}$ has been absorbed into $\phi$.
\begin{equation}
\sum_{i=1}^{n}i^{\frac{1}{r}}=\frac{r}{r+1}(n+1)^{\frac{1+r}{r}}-\frac{1}{2}(n+1)^{\frac{1}{r}}-\phi
\end{equation}

\ Taking limit of (2) as $r\rightarrow \infty$,$L.H.S.\rightarrow n$ and $R.H.S.\rightarrow (n+\frac{1}{2}-\phi)$, so that in the limit  $\phi\rightarrow \frac{1}{2}$\\
\ On the other extreme, for $r=1$,\\
\begin{eqnarray*}
R.H.S.&=&\frac{1}{2}(n+1)^{2}-\frac{1}{2}(n+1)-\phi\\
&=&\frac{1}{2}(n^{2}+n)-\phi\\
&=&\frac{n(n+1)}{2}-\phi\\
\end{eqnarray*}
and,$$L.H.S.=\frac{n(n+1)}{2}$$
This gives $\phi_{n}(1)=0$\\

Since the difference between first and second term can easily be shown to be monotonic, we see that
$\phi$ is bounded between between 0 and $\frac{1}{2}$ for $1\leq r < \infty $
\end{proof}

As an application of the formula derived above, we derive formula to derive factorial of positive integer. We begin by taking derivative of (2) with respect to $r$. After rearranging the terms, we get,
\begin{equation}
\sum_{i=1}^{i=n}i^{\frac{1}{r}}\log{i}=\left[\frac{r}{r+1}(n+1)-\frac{1}{2}\right](n+1)^{\frac{1}{r}}\log(n+1)-\frac{r^{2}}{(r+1)^{2}}(n+1)^{\frac{1+r}{r}}+r^{2}\frac{d\phi}{dr}
\end{equation}
After taking limit of this equation as $r\rightarrow \infty$, we get following equation:
\begin{equation}
\sum_{i=1}^{n}\log{i}=(n+\frac{1}{2})\log(n+1)-(n+1)+\lim_{r\to\infty}r^{2}\frac{d\phi}{dr}
\end{equation}

In the above expression, $L.H.S$ is just $\log(n!)$. Let us assume that limit in the last term of the above equation exists and is finite and say that it is $\xi$. Then we can rewrite above equation as follows:
\begin{equation}
n!=(n+1)^{n+\frac{1}{2}}e^{-n-1}e^{\xi}
\end{equation}

Numerically it turns out that the quantity $e^{\xi}$ indeed converges to finite value, the value being close to $\sqrt{2\pi}$. This formula is similar to precise version of Stirling's formula (\cite{St}).

Equation (3) allows us to find one more interesting formula. After putting $r=1$ in (3) and after little rearrangement, we get following beautiful formula:
\begin{equation}
\log\left[1^{1}.2^{2}... n^{n}\right]=\frac{n(n+1)}{2}\log(n+1)-\frac{1}{4}(n+1)^{2}+\frac{d\phi}{dr}|_{r=1}
\end{equation}

Numerically it turns out that quantity $\frac{d\phi}{dr}|_{r=1}$ is very small and can be neglected.

\end{document}